\documentclass[11pt]{article}
\usepackage{amssymb}
\usepackage{amsmath}
\usepackage{amsfonts}
\usepackage{amsthm}
\usepackage{amsmath}

\usepackage{tikz}

\usepackage[english]{babel}
\usepackage[affil-it]{authblk}

  \makeatletter
\def\@fnsymbol#1{\ensuremath{\ifcase#1\or  \ddagger\or
		\mathsection\or \mathparagraph\or \|\or **\or \dagger\dagger
		\or \ddagger\ddagger \else\@ctrerr\fi}}
\makeatother

\newtheorem{theorem}{Theorem}

\newtheorem{problem}{Problem}
\newtheorem{proposition}{Proposition}
\newtheorem{corollary}{Corollary}

\newtheorem{conjecture}{Conjecture}

\providecommand{\keywords}[1]{\textbf{\textit{Keywords:}} #1}

\title{On the palindromic Hosoya polynomial of trees}
	
\author[1]{Dmitry Badulin}
\author[2]{Alexandr Grebennikov}
\author[3]{Konstantin Vorob'ev}

\affil[1]{Lomonosov Moscow State University, Moscow, Russia; E-mail: badulind@bk.ru}
\affil[2]{St. Petersburg State University, Saint-Petersburg, Russia; E-mail: sagresash@yandex.ru}
\affil[3]{Sobolev Institute of Mathematics, Novosibirsk, Russia; E-mail: konstantin.vorobev@gmail.com}

\begin{document}

\maketitle

\begin{abstract}

A graph $G$ on $n$ vertices of diameter $D$ is called $H$-palindromic if $\alpha(G,k) = \alpha(G,D-k)$ for all $k=0, 1, \dots, \left \lfloor{\frac{D}{2}}\right \rfloor$, where $\alpha(G,k)$ is the number of unordered pairs of vertices at distance $k$.
Quantities $\alpha(G,k)$ form coefficients of the Hosoya polynomial.
In 1999, Caporossi, Dobrynin, Gutman and Hansen showed that there are exactly five $H$-palindromic trees of even diameter and conjectured that there are no such trees of odd diameter. We prove this conjecture for bipartite graphs. An infinite family of $H$-palindromic trees of diameter $6$
is also constructed.

\end{abstract}



\section{Introduction}\label{S:Introduction}

Let $G=(V,E)$ be an undirected  connected graph without loops and multiple edges.
The distance $d(x,y)$ between vertices $x,y\in V$ is the number of edges in the shortest path connecting $x$ and $y$ in $G$.
The maximal distance between vertices of a graph is called its diameter $D$.
The Hosoya polynomial of a graph $G$ of diameter $D$ is defined as
$$
H(G,\lambda)=\sum_{k=0}^{D} \alpha(G,k)\lambda^k,
$$
where $\alpha(G,k)$ is equal to the number of unordered pairs of vertices at distance $k$ in $G$.
Clearly, $\alpha(G,0)=|V|$ and $\alpha(G,1)=|E|$.
This polynomial was first proposed by Hosoya under the name Wiener polynomial in 1988 \cite{zbMATH04029576}.
It was studied for various classes of abstract and molecular graphs.
Historical remarks and the bibliography on the Hosoya polynomial can be found in
\cite{GZDI2012}.
The Wiener index $W(G)$ is a distance-based graph invariant defined as the sum
of distances over all unordered pairs of vertices of a graph $G$.
Therefore, it can be presented through the coefficients of the Hosoya polynomial as follows
$$
W(G)=\sum_{k=1}^{D} \alpha(G,k) k,
$$
that is the Wiener index can be calculated as the first derivative of $H(G, \lambda)$ at $\lambda=1$.
This index was introduced by Harry Wiener for molecular graphs of alkanes that are trees in 1947 \cite{W1947}.
It has numerous applications in organic chemistry (see, for example, reviews \cite{zbMATH01656310,zbMATH01783464,zbMATH06680393}).

A graph $G$ is $H$-palindromic if the
equality $\alpha(G,k) = \alpha(G,D-k)$ holds for all $k=0, 1, \dots, \left \lfloor{\frac{D}{2}}\right \rfloor$.
The quantity
$$
Z(G)=\sum_{k=0}^{\left \lfloor{\frac{D}{2}}\right \rfloor}{\large| \alpha(G,k)-\alpha(G,D-k) \large| }
$$
is called the distance to $H$-palindromicity of a graph $G$.
Clearly, a graph is $H$-palindromic if and only if its distance to $H$-palindromicity equals $0$.
It is known that the Wiener index of $H$-palindromic trees $T$ depends
only on the number of vertices $n$ and the diameter $D$:
$W(T)=D\frac{n(n+1)}{4}$ \cite{CDGH1999}.

Some families of $H$-palindromic cyclic graphs have been constructed in \cite{D1994GTNNY}.
Based on computer experiments, Gutman conjectured that there are no $H$-palindromic trees \cite{G1993}
(see also \cite{GEI1999}). Only five palindromic trees were found by exhaustive computer search among 
all trees with $n \leq 26$ vertices \cite{CDGH1999}.
The following table shows the number of vertices, diameter and coefficients of the palindromic Hosoya polynomial of these trees.

\begin{table}[h]
	\centering
	\begin{tabular}{ c|c|c|c}
		\hline
		$T$ & $n$ & $D$ & $\left(\alpha(T,0), \alpha(T,1), \dots, \alpha(T,D)\right)$ \\ \hline
		$T_1$  & $21$ & $8$ & $(21, 20, 34, 25, 31, 25, 34, 20, 21)$ \\
		$T_2$  & $22$ & $6$ & $(22, 21, 52, 63, 52, 21, 22)$ \\
		$T_3$  & $22$ & $6$ & $(22, 21, 52, 63, 52, 21, 22)$ \\
		$T_4$  & $24$ & $8$ & $(24, 23, 39, 41, 46, 41, 39, 23, 24)$ \\
		$T_5$  & $24$ & $8$ & $(24, 23, 37, 41, 50, 41, 37, 23, 24)$ \\
		\hline
	\end{tabular}
	\caption{$H$-palindromic trees of diameter $D$ with $n$ vertices.}
\end{table}

Some necessary conditions for the existence of $H$-palindromic trees of odd diameter were found
and the following conjectures were formulated in \cite{CDGH1999} (see also \cite{zbMATH01656310}).

\begin{conjecture}{\label{Con1}}
For all trees with $n>4$ vertices and odd diameter the distance to $H$-palindromicity 
is at least $\left \lceil{\frac{n}{2}}\right \rceil$.
\end{conjecture}

\begin{conjecture}{\label{Con2}}
	There are no $H$-palindromic trees of odd diameter.
\end{conjecture}
Evidently, the second conjecture is a consequence of the first one.
An intensive computations were done to test Conjecture \ref{Con1} in \cite{zbMATH02041897}.
So far no progress has been made on this problem.

In this work, we prove these conjectures for bipartite graphs.
We also prove that there are infinitely many $H$-palindromic trees of diameter $6$.

\section{Trees of odd diameter}\label{S:Odd}
In this Section, we consider bipartite graphs of odd diameter.

\begin{theorem}{\label{T:Odd}}
	Let $G$ be a bipartite graph on $n$ vertices of odd diameter. Then
$$
Z(G) \ge \left \lceil{\frac{n}{2}}\right \rceil .
$$
\end{theorem}
\begin{proof}
Let $a$ and $b$ be the cardinalities of the bipartite parts of a graph $G$
of diameter $D$.
Obviously, the distance between two vertices of $G$ is even if and only if they belong to the same part.
Hence, the sum of $\alpha(G,i)$ over odd $i$ equals the number of pairs from different parts:
$$
\sum_{\substack{i=0\\i \text{ is odd}}}^{D} \alpha(G,i) = ab,	
$$
and the sum of $\alpha(G,i)$ over even $i$ equals the number of pairs from the first part and pairs from the second part:
$$
\sum_{\substack{i=0\\i \text{ is even}}}^{D} \alpha(G,i) = {a\choose 2} + a + {b\choose 2} + b = \frac{a^2 + a + b^2 + b}{2}.
$$
Since the diameter $D$ is odd, quantities $i$ and $D - i$ have different parity. Then
$$
 Z(G) \ge \sum_{\substack{i=0\\i \text{ is even}}}^{D}{\alpha(G,i)} - 	\sum_{\substack{i=0\\i \text{ is odd}}}^{D}{\alpha(G,i)} =
             \frac{(a-b)^2 + a + b}{2} \ge \frac{a+b}{2} = \frac{n}{2}.
$$
By definition, $Z(G)$ is an integer, so the claim follows.
\end{proof}

Since an arbitrary tree is a bipartite graph,  we immediately have the following result.
\begin{corollary}
There are no $H$-palindromic trees of odd diameter.
\end{corollary}

The bound of Theorem \ref{T:Odd} is sharp. For instance,
consider the Hamming graph $H(m,2)$ of order $2^m$.
Its vertex set consists of all binary words of length $m$ with the usual Hamming distance.

\begin{proposition}{\label{P:Hamm}}
For the Hamming graph $H(m,2)$, $Z(H(m,2))=2^{m-1}$.
\end{proposition}
\begin{proof}
By definition, the diameter of $H(m,2)$ is equal to $m$.
Every vertex of the graph has ${m \choose k}$ neighbors at distance $k$ for $0 \leq k \leq m$. 	
Hence,  $\alpha(G,k) = \alpha(G,m-k)$ for $1 \leq k \leq m-1$,
$\alpha(G,0)=2^m$, and $\alpha(G,m)=2^{m-1}$. Then
$$
Z(H(m,2))=\sum_{k=1}^{\left \lfloor{\frac{m}{2}}\right \rfloor}{\biggl| {m \choose k}-{m \choose m-k} \biggr|}+ |2^m-2^{m-1} | =2^{m-1},
$$
that is a half of the number of vertices of $H(m,2)$.
\end{proof}

\section{Trees of diameter $6$}\label{S:even}
As it was discussed in Section \ref{S:Introduction}, only five $H$-palindromic trees of even diameter are known.
It is easy to show that there are no such trees of diameter $2$ and $4$.
It is sufficient to consider a general model of such a tree and find coefficients of Hosoya polynomial
by direct calculations. However, trees of diameter $6$ may be $H$-palindromic.

\begin{theorem}
	There is an infinite number of $H$-palindromic trees of diameter $6$.
\end{theorem}
\begin{proof}
	For non-negative integers $a$, $b$, $s$ and $t$, construct a tree $T = T(a, b, s, t)$ by the following steps:
	\begin{enumerate}
		\item take a path of length five: $(v_1,v_2,v_3,v_4,v_5,v_6)$,
		\item attach one pendent vertex to vertex $v_2$ and one pendent vertex $u$ to vertex $v_5$,
		\item attach $t$, $s$, $a$ and $b$ new pendent vertices to vertices $v_4$, $v_5$, $v_6$ and $u$, respectively (see Fig.~1).
	\end{enumerate}
	
\begin{center}	
	\begin{tikzpicture}[thick, scale=1.4, every node/.style={scale=0.9}]
		\tikzstyle{n}=[draw,circle,fill=black,minimum size=0pt, text width=0.0cm,
		inner sep=1.5pt]
		\tikzstyle{e}=[inner sep=0pt, outer sep=0pt]
		\node[n] at (1, 1){};
		\node[n, label={$v_1$}] at (1, 3){};
		\node[n, label={$v_2$}] at (2.5, 2){};
		\node[n, label={$v_3$}] at (4, 2){};
		\node[n, label=270:{$v_4$}] at (5.5, 2){};
		\node[n, label={$v_5$}] at (7, 2){};
		\node[n, label={$v_6$}] at (8.5, 1){};
		\node[n, label={$u$}] at (8.5, 3){};
		
		\node[n] at (6.3, 3){};
		\node[n] at (4.7, 3){};
		\node[n] at (5.1, 3){};
		\node[e] at (5.7, 3){$\ldots$};
		
		\node[n] at (5.9, 0.5){};
		\node[n] at (6.3, 0.5){};
		\node[n] at (7.2, 0.5){};
		\node[e] at (6.7, 0.5){$\ldots$};
		
		\node[n] at (9.5, 2.4){};
		\node[n] at (9.5, 2.7){};
		\node[n] at (9.5, 3.6){};
		\node[e] at (9.5, 3.2){$\vdots$};
		
		\node[n] at (9.5, 0.4){};
		\node[n] at (9.5, 0.7){};
		\node[n] at (9.5, 1.6){};
		\node[e] at (9.5, 1.2){$\vdots$};
		
		\draw (5.5, 3) ellipse (1cm and 0.4cm);
		\node[e] at (5.5, 3.6){\large $t$};
		\draw (6.5, 0.5) ellipse (1cm and 0.4cm);
		\node[e] at (6.5, -0.2){\large $s$};
		\draw (9.5, 1) ellipse (0.4cm and 0.8cm);
		\node[e] at (10.2, 1){\large $a$};
		\draw (9.5, 3) ellipse (0.4cm and 0.8cm);
		\node[e] at (10.2, 3){\large $b$};
		
		\draw (1, 1) -- (2.5, 2);
		\draw (1, 3) -- (2.5, 2);
		\draw (4, 2) -- (2.5, 2);
		\draw (4, 2) -- (5.5, 2);
		\draw (7, 2) -- (5.5, 2);
		\draw (7, 2) -- (8.5, 1);
		\draw (7, 2) -- (8.5, 3);
		\draw (5.5, 2) -- (4.7, 3);
		\draw (5.5, 2) -- (5.1, 3);
		\draw (5.5, 2) -- (6.3, 3);
		\draw (7, 2) -- (5.9, 0.5);
		\draw (7, 2) -- (6.3, 0.5);
		\draw (7, 2) -- (7.2, 0.5);
		\draw (8.5, 1) -- (9.5, 1.6);
		\draw (8.5, 1) -- (9.5, 0.7);
		\draw (8.5, 1) -- (9.5, 0.4);
		\draw (8.5, 3) -- (9.5, 2.4);
		\draw (8.5, 3) -- (9.5, 2.7);
		\draw (8.5, 3) -- (9.5, 3.6);
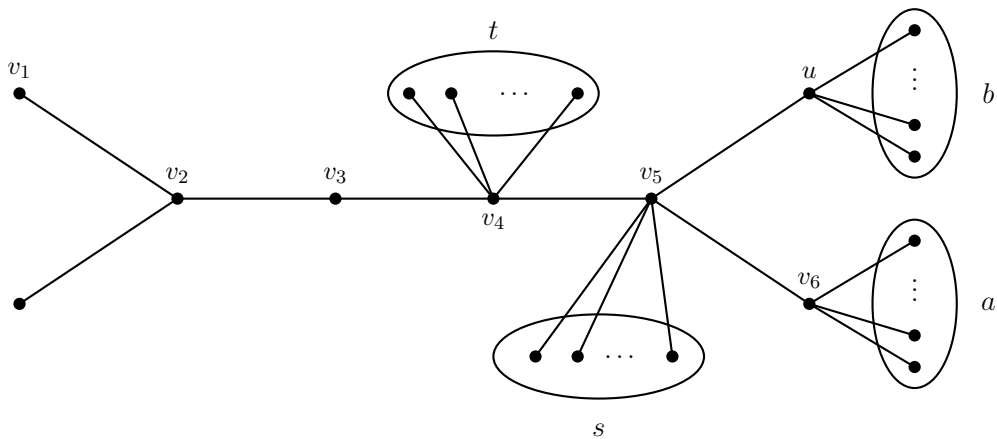
	\node[below,font=\large\bfseries] at (current bounding box.south) {Figure 1. Construction of $H$-palindromic tree $T(a,b,s,t)$};	
	\end{tikzpicture}
\end{center}	

Counting pairs of vertices at a given distance in the tree of diameter 6 in Fig.~1,
one can calculate values of coefficients $\alpha(T,i)$:

$$
	\begin{cases}
		  \alpha(T,0) = a + b + s + t + 8,   \\
		  \alpha(T,1) = a + b + s + t + 7,   \\
		  \alpha(T,2) = {a+1\choose2} + {b+1\choose2} + {s+3\choose2} + {t+2\choose2} + 4,   \\
		  \alpha(T,4) = (s + 2) + (a + b + 2)(t + 1) + ab,   \\
		  \alpha(T,5) = a + b + 2(s + 2),   \\
		  \alpha(T,6) = 2(a + b).
	\end{cases}
$$
	
Equalities $\alpha(T,0) = \alpha(T,6)$ and $\alpha(T,1) = \alpha(T,5)$ are
satisfied under condition $s = t + 3 = \frac{a + b - 5}{2}$.
Since $s$ and $t$ are non-negative integers, the sum $a+b$ should be odd and not less than $11$.
So it remains to satisfy the equation $\alpha(T,2) = \alpha(T,4)$.
After all necessary calculations, one can rewrite this equality in the following form:

$$
(a - 3b + 3)^2 - 2(2b - 3)^2 + 94 = 0.
$$

Using substitution $x = a - 3b + 3$ and $y = 2b - 3$, the last equality can be presented as the Pell equation
$$
x^2 - 2y^2 = -94.
$$
Methods for solving Pell equation can be found in \cite{zbMATH00886138}.
It has an infinite series of integer solutions starting from  $(x, y) = (2, 7)$.
All solutions can be represented by the following recurrent relations:
$$
\begin{cases}
x_{n+1} = 3x_n + 4y_n \\
y_{n+1} = 3y_n + 2x_n
\end{cases}
$$
with the initial conditions $x_0 = 2$ and $y_0 = 7$.
It easy to see that $x_n$ is always even and $y_n$ is odd.
Therefore $a_n = x_n + \frac{3y_n + 3}{2}$ and $b_n = \frac{y_n + 3}{2}$ are both integers
and their sum $a_n + b_n$ is odd and not less than $a_0 + b_0 = 19 \ge 11$.
So, every pair $(a_n, b_n)$ corresponds to some $H$-palindromic tree.
\end{proof}

Table~\ref{T2} shows solutions of the Pell equation and parameters of the initial part of the constructed
series of $H$-palindromic vertex trees $T$ of diameter 6.

\begin{table}[htb] \label{T2}
\centering
\begin{tabular}{rrrrrrrrc}
\hline
$n$ &$x_n$ &$y_n$ & $|V|$ & $a$ & $b$ & $s$ & $t$  &  coefficients of $H(T,\lambda)$\\ \hline
 0  &   2  &  7   &  38    & 14   &  5  &  7  &  4   &  (38, 37, 184, 223, 184, 37, 38)      \\
 1  &   34  &  25   &  174    & 73   &  14  &  41  &  38   &  (174, 173, 4536, 5459, 4536, 173, 174) \\
 2  &   202  &  143   &  982    & 418   &  73  &  243  &  240   &  (982, 981, 149572, 179583, 149572, 981, 982) \\  
 3  &   1178  &  833   &  5694    & 2429   &  418  &  1421  &  1418   &  (5694, 5693, 5059476, 6071939, 5059476, 5693, 5694) \\  \hline
\end{tabular}
\caption{First $H$-palindromic trees $T$ of diameter 6.}
\end{table}

It will be interesting to answer the following question.

\begin{problem}
Does there exist an infinite family of $H$-palindromic trees of even diameter $D\ge 8$?
\end{problem}

We suppose that ideas from Section \ref{S:even} may be successfully applied for small values of $D$.

\section{Acknowledgements}

This research project was started during the Summer Research Program for Undergraduates 2021 organized by the Laboratory of Combinatorial and Geometric Structures at MIPT. This program was funded by the Russian Federation Government (Grant number 075-15-2019-1926). The work of Konstantin Vorob'ev was carried out within the framework of the state contract of the Sobolev Institute of Mathematics (project no. 0314-2019-0016). The work of Alexandr Grebennikov is supported by Ministry of Science and Higher Education of the Russian Federation,
agreement no. 075-15-2019-1619. Authors are grateful to Andrey Dobrynin for interesting discussions on the theme of this paper.

\bibliography{bib_Hosoya}

\providecommand\href[2]{#2} \providecommand\url[1]{\href{#1}{#1}}
  \def\DOI#1{{\small {DOI}:
  \href{http://dx.doi.org/#1}{#1}}}\def\DOIURL#1#2{{\small{DOI}:
  \href{http://dx.doi.org/#2}{#1}}}
\begin{thebibliography}{10}

\bibitem{CDGH1999}
G.~{Caporossi}, A.A. {Dobrynin}, I.~{Gutman}, and P.~{Hansen}.
\newblock {Trees with palindromic Hosoya polynomials}.
\newblock {\em {Graph Theory Notes N. Y.}}, 37:10--16, 1999.

\bibitem{zbMATH02041897}
G.~{Caporossi} and P.~{Hansen}.
\newblock {Variable neighborhood search for extremal graphs. V: Three ways to
  automate finding conjectures}.
\newblock {\em {Discrete Math.}}, 276(1-3):81--94, 2004.

\bibitem{D1994GTNNY}
A.~A. {Dobrynin}.
\newblock {Graphs with palindromic Wiener polynomials}.
\newblock {\em {Graph Theory Notes N. Y.}}, 27:50--54, 1994.

\bibitem{zbMATH01656310}
A.~A. {Dobrynin}, R.~{Entringer}, and I.~{Gutman}.
\newblock {Wiener index of trees: Theory and applications}.
\newblock {\em {Acta Appl. Math.}}, 66(3):211--249, 2001.

\bibitem{zbMATH01783464}
A.~A. {Dobrynin}, I.~{Gutman}, S.~{Klav\v{z}ar}, and P.~{\v{Z}igert}.
\newblock {Wiener index of hexagonal systems}.
\newblock {\em {Acta Appl. Math.}}, 72(3):247--294, 2002.

\bibitem{G1993}
I.~{Gutman}.
\newblock {Some properties of the Wiener polynomial}.
\newblock {\em {Graph Theory Notes N. Y.}}, 25:13--18, 1993.

\bibitem{GEI1999}
I.~{Gutman}, E.~{ Estrada}, and O.~{Ivanciuc}.
\newblock {Some properties of the Wiener polynomial of treesl}.
\newblock {\em {Graph Theory Notes N. Y.}}, 36:7--13, 1999.

\bibitem{GZDI2012}
I.~{Gutman}, Y.~{Zhang}, M.~{Dehmer}, and A.~{Ili\'{c}}.
\newblock Altenburg, {W}iener, and {H}osoya polynomials.
\newblock In I.~{Gutman} and B.~{Furtula}, editors, {\em Distance in Molecular
  Graphs --- Theory}, pages 49--70. Univ. Kragujevac, Kragujevac, 2012.

\bibitem{zbMATH04029576}
H.~{Hosoya}.
\newblock {On some counting polynomials in chemistry}.
\newblock {\em {Discrete Appl. Math.}}, 19:239--257, 1988.

\bibitem{zbMATH06680393}
M.~{Knor}, R.~{\v{S}krekovski}, and A.~{Tepeh}.
\newblock {Mathematical aspects of Wiener index}.
\newblock {\em {Ars Math. Contemp.}}, 11(2):327--352, 2016.

\bibitem{zbMATH00886138}
Don {Redmond}.
\newblock {\em {Number theory. An introduction}}.
\newblock Basel: Marcel Dekker, 1996.

\bibitem{W1947}
H.~{Wiener}.
\newblock {Structural Determination of Paraffin Boiling Points}.
\newblock {\em {J. Am. Chem. Soc.}}, 69(1):17--20, 1947.

\end{thebibliography}

\end{document}